\documentclass[10pt]{article}

\usepackage[margin=1.12in]{geometry}
\usepackage[T1]{fontenc}
\usepackage[utf8]{inputenc}
\usepackage{amsmath,amssymb,amsthm,mathtools}
\usepackage{authblk}
\usepackage{enumitem}
\usepackage{aliascnt}
\usepackage[hyperfootnotes=false]{hyperref}
\usepackage[nameinlink,noabbrev]{cleveref}
\usepackage{microtype}

\hypersetup{colorlinks=true,linkcolor=blue,citecolor=blue,urlcolor=blue}

\setlist[enumerate]{leftmargin=2.1em,itemsep=0.15em,topsep=0.25em}
\allowdisplaybreaks[2]

\theoremstyle{plain}
\newtheorem{theorem}{Theorem}[section]
\newaliascnt{lemma}{theorem}
\newtheorem{lemma}[lemma]{Lemma}
\aliascntresetthe{lemma}
\newaliascnt{proposition}{theorem}
\newtheorem{proposition}[proposition]{Proposition}
\aliascntresetthe{proposition}
\newaliascnt{corollary}{theorem}

\aliascntresetthe{corollary}
\newaliascnt{conjecture}{theorem}

\aliascntresetthe{conjecture}
\newaliascnt{question}{theorem}

\aliascntresetthe{question}
\newaliascnt{claim}{theorem}
\newtheorem{claim}[claim]{Claim}
\aliascntresetthe{claim}
\newaliascnt{definition}{theorem}

\aliascntresetthe{definition}
\theoremstyle{remark}
\newaliascnt{remark}{theorem}

\aliascntresetthe{remark}
\newtheorem*{remark*}{Remark}

\crefname{theorem}{Theorem}{Theorems}
\crefname{lemma}{Lemma}{Lemmas}
\crefname{proposition}{Proposition}{Propositions}
\crefname{corollary}{Corollary}{Corollaries}
\crefname{conjecture}{Conjecture}{Conjectures}
\crefname{question}{Question}{Questions}
\crefname{claim}{Claim}{Claims}
\crefname{definition}{Definition}{Definitions}
\crefname{remark}{Remark}{Remarks}
\crefname{section}{Section}{Sections}
\crefname{subsection}{Subsection}{Subsections}
\crefname{equation}{Equation}{Equations}
\crefalias{enumi}{proposition}

\newcommand{\ex}{\operatorname{ex}}
\newcommand{\codeg}{\operatorname{codeg}}

\newcommand{\tr}{t}

\title{\Large \bf On Mubayi's Polynomial-Ideal Conjecture and a Hypergraph Tur\'{a}n Theorem}
\author[1,2]{Heng~Li\thanks{Email:\texttt{heng.li@sdu.edu.cn}}}
\author[3]{Xizhi~Liu\thanks{Email:\texttt{liuxizhi@ustc.edu.cn}}}
\affil[1]{\small School of Mathematics, Shandong University, Jinan, China}
\affil[2]{\small Extremal Combinatorics and Probability Group, Institute for Basic Science, Daejeon, South Korea}
\affil[3]{\small School of Mathematical Sciences, University of Science and Technology of China, Hefei, China}
\date{\today}

\begin{document}
\maketitle

\begin{abstract}
Among the many proofs of Tur\'{a}n's classical theorem, one particularly surprising proof due to Li and Li uses ideals in polynomial rings to record missing edges. Motivated by their proof, Mubayi proposed a hypergraph analogue, conjecturing that an ideal generated by multipartite $3$-graphs coincides with a differentiated diagonal-vanishing ideal. If true, this conjecture would imply the extremal-number part of Mubayi's classical hypergraph extension of Tur\'{a}n's theorem. We disprove this conjecture throughout the nontrivial parameter range. We then give an alternative algebraic proof of Mubayi's extremal formula for the family $\mathcal{K}_{\ell}^{(r)}$ of clique expansions, using monomial cover ideals and a Hilbert-function symmetrization theorem for square-zero quadratic monomial quotients.
\end{abstract}



\section{Introduction}

Given a family $\mathcal F$ of $r$-uniform hypergraphs (\emph{$r$-graphs} for short) and an $r$-graph $\mathcal{G}$, we say that $\mathcal{G}$ is \emph{$\mathcal F$-free} if it contains no member of $\mathcal F$ as a subhypergraph. The \emph{extremal number} $\ex(n,\mathcal F)$ is the maximum value of $|\mathcal{G}|$ over all $\mathcal F$-free $r$-graphs $\mathcal{G}$ on $n$ vertices. In the special case that $\mathcal F=\{F\}$, we write $\ex(n,F)$.

The case $r=2$ is the classical starting point. Tur\'{a}n's theorem determines $\ex(n,K_{q+1})$ exactly~\cite{Turan}, and the Erd\H{o}s--Stone theorem gives the asymptotic answer for every non-bipartite forbidden graph~\cite{ErdosStone}. Variants in which one maximizes the number of copies of a fixed graph also go back to Zykov's symmetrization and clique-counting theorem~\cite{Zykov}; this viewpoint is now part of the generalized Tur\'{a}n problem studied systematically by Alon and Shikhelman~\cite{AlonShikhelman}. For $r\ge 3$, the picture is much less settled. Even the natural analogues for complete hypergraphs are difficult, and the exact results that are known typically require additional structure. For general background on hypergraph Tur\'{a}n theory, we refer to the surveys of Sidorenko~\cite{SidorenkoSurvey} and Keevash~\cite{KeevashSurvey}.

Graph expansions provide one important source of such structured hypergraphs. Given a graph $F$, its $r$-uniform expansion is obtained by enlarging each edge of $F$ with new vertices, using disjoint new vertex sets for distinct edges. The expansion of a clique is a particularly natural hypergraph substitute for a complete graph. Mubayi considered the corresponding weak expansions of $K_\ell$: one chooses an $\ell$-vertex core and, for each pair of core vertices, an $r$-edge containing that pair. As a way to extend Tur\'{a}n's theorem to hypergraphs, Mubayi determined the Tur\'{a}n number for $\mathcal{K}_{\ell}^{(r)}$, the family of all weak expansions of $K_\ell$; in particular, his result provided the first infinite sequence of Tur\'{a}n densities for every $r\ge 3$~\cite{Mubayi}. Pikhurko later obtained the exact Tur\'{a}n number for the ordinary $r$-uniform expansions of complete graphs for all sufficiently large $n$~\cite{Pikhurko}. Stability and exact results for related expansions and extensions have since been developed in several directions; see, for instance, Liu~\cite{LiuStability}, Brandt--Irwin--Jiang~\cite{BrandtIrwinJiang}, Norin--Yepremyan~\cite{NorinYepremyan}, and Liu--Mubayi--Reiher~\cite{LiuMubayiReiher}. More broadly, Tur\'{a}n problems for expansions have become an active topic; see the survey of Mubayi and Verstra\"{e}te~\cite{MubayiVerstraeteSurvey} and the work of Kostochka, Mubayi, and Verstra\"{e}te on expansions and shadows~\cite{KostochkaMubayiVerstraete}.

This paper starts from a different but related motivation: an algebraic proof of Tur\'{a}n's theorem due to Li and Li~\cite{LiLi}. Recall that an \emph{ideal} $I$ of a polynomial ring $R$ is a collection of polynomials closed under addition and under multiplication by arbitrary elements of $R$. If $f_1,\ldots,f_m\in R$, then the ideal generated by them is
\[
\langle f_1,\ldots,f_m\rangle
=
\Big\{\sum_{i\in[m]}h_i f_i\colon h_i\in R\text{ for every }i\Big\}.
\]
Let $\mathbb{K}$ be a field of characteristic zero, and set $\mathrm{S}_n\coloneqq\mathbb{K}[x_1,\ldots,x_n]$. For $s\ge 2$ and $L\in\binom{[n]}s$, write $\Delta_L$ for the \emph{diagonal} $x_i=x_j$ for all $i,j\in L$, and set
\[
\mathrm{I}(n,s)\coloneqq \left\{ f\in \mathrm{S}_n\colon f|_{\Delta_L}=0\text{ for every }L\in\tbinom{[n]}s \right\}.
\]
For a graph $G$ on vertex set $[n]$, define
\[
p_G(x_1,\ldots,x_n)\coloneqq
\prod_{\substack{i<j\\ \{i,j\}\notin G}}(x_i-x_j).
\]
An $s$-set $L$ spans a clique in $G$ if and only if no factor $x_i-x_j$ with $i,j\in L$ occurs in $p_G$, and hence if and only if $p_G|_{\Delta_L}\not\equiv 0$. Thus $G$ is $K_s$-free if and only if $p_G\in \mathrm{I}(n,s)$.
Since $\deg p_G=\binom n2-|G|$, lower bounds on the degrees of nonzero elements of $\mathrm{I}(n,s)$ translate into upper bounds on $|G|$, giving the Tur\'{a}n bound.
For a combinatorial reader, the important point is that membership in $\mathrm{I}(n,s)$ packages all $s$-vertex clique constraints into one algebraic condition, while the degree of $p_G$ is exactly the number of missing edges.
Thus the algebra is used to prove that any polynomial satisfying the diagonal-vanishing conditions must have sufficiently large degree, which is the same as saying that $G$ must have sufficiently many missing edges.

Mubayi observed in~\cite{Mubayi} that this proof does not immediately extend to hypergraphs. The difficulty is already visible for 3-graphs: the relevant missing object is not a single missing pair, but a whole missing star through a pair. To capture this codegree-zero condition, Mubayi proposed a differentiated diagonal-vanishing ideal. For a 3-graph $\mathcal{G}$ on $[n]$, set 
\[
p_{\mathcal{G}}(x_1,\ldots,x_n)\coloneqq
\prod_{\substack{i<j<k\\ \{i,j,k\}\notin \mathcal{G}}}
(x_i-x_j)(x_i-x_k)(x_j-x_k).
\]
Mubayi's \emph{differentiated diagonal-vanishing ideal} is
\[
\mathrm{DI}(n,\ell)\coloneqq
\left\{
f\in \mathrm{S}_n\colon
\tfrac{\partial^j f}{\partial x_i^j}\in \mathrm{I}(n,\ell)
\text{ for every }i\in[n]\text{ and every }0\le j\le n-3
\right\}.
\]
On the other side, let
\[
\widehat{\mathrm{P}}^{(3)}(n,\ell)\coloneqq
\bigl\langle p_{\mathcal{G}}\colon
\mathcal{G}\text{ is an }(\ell-1)\text{-partite }3\text{-graph on }[n]\bigr\rangle
\subseteq \mathrm{S}_n.
\]
In \cite[Conjecture 6]{Mubayi}, Mubayi conjectured that $\widehat{\mathrm{P}}^{(3)}(n,\ell)=\mathrm{DI}(n,\ell)$. If true, this would imply the hypergraph Tur\'{a}n upper bound for the family $\mathcal{K}_{\ell}^{(3)}$ introduced below.

With this notation, the failure of the proposed equality has the following simple form.

\begin{theorem}\label{thm:intro-counterexample}
For every $\ell\ge 3$ and every $n\ge \ell$, we have $\widehat{\mathrm{P}}^{(3)}(n,\ell)\subsetneq \mathrm{DI}(n,\ell)$.
\end{theorem}

When $n<\ell$, both ideals are the whole polynomial ring $\mathrm{S}_n$. Thus the strict containment in \cref{thm:intro-counterexample} holds precisely in the nontrivial range.

The counterexamples are elementary Vandermonde-type products, but they expose a structural obstruction. The differentiated diagonal conditions can be satisfied by different pair factors on different diagonals; they do not force a fixed pair to have its entire missing star. Thus $\mathrm{DI}(n,\ell)$ loses exactly the combinatorial information needed for the intended hypergraph application.

The second part of the paper replaces this diagonal-vanishing language with a monomial model that records missing edges directly. We work in an \emph{edge-variable ring}, with one variable for each possible hyperedge. Forbidden copies give a \emph{Stanley--Reisner ideal}, and its Alexander dual is the corresponding \emph{cover ideal}. This construction is standard in combinatorial commutative algebra~\cite{Stanley,Reisner,EagonReiner,MillerSturmfels,HerzogHibi}; in the present setting, it turns hitting all forbidden copies into a divisibility condition on missing-edge monomials.

For ordinary extremal problems, this translation is exact $\alpha(C_{\mathcal F,n})=\binom nr-\ex(n,\mathcal F)$, 
where $\alpha$ denotes initial degree. For generalized Tur\'{a}n numbers in the sense of Alon and Shikhelman~\cite{AlonShikhelman}, one forbids a graph $F$ but maximizes copies of another graph $T$. In that setting, ordinary degree is replaced by a quotient-theoretic rank on the vector space spanned by the $T$-copy monomials, giving $\ex(n,T,F)=|\mathcal T_{T,n}|-\alpha_T(C_{\mathcal F,n})$. 
Here $\alpha_T$ measures the minimum dimension of the part of the $T$-copy space killed by a missing-edge quotient; when $T=K_2$, it reduces to the usual initial degree.

Finally, we apply this framework to Mubayi's clique-expansion family $\mathcal{K}_{\ell}^{(r)}$. For this family the cover ideal collapses to a missing star ideal. We compute its initial degree by translating its monomials into square-zero quadratic monomial quotients and applying a Hilbert-function symmetrization theorem, an algebraic version of the Zykov symmetrization. This gives a monomial-ideal proof of Mubayi's extremal formula $\ex(n,\mathcal{K}_{\ell}^{(r)})=\tr_r(n,\ell-1)$.

\paragraph*{Organization of the paper.}
\cref{sec:prelim} collects the notation and algebraic background used throughout the paper. \cref{sec:counterexamples} proves the strict containment in \cref{thm:intro-counterexample}. \cref{sec:cover-ideal-method} develops the cover-ideal method: \cref{sec:hilbert-symmetrization} proves the Hilbert-function symmetrization theorem for square-zero quadratic monomial quotients, \cref{sec:cover-ideals} translates ordinary and generalized Tur\'{a}n problems into cover-ideal language, and \cref{sec:star-ideal} computes the missing star ideal and derives Mubayi's extremal formula.

\section{Preliminaries}\label{sec:prelim}

Throughout, $[n]\coloneqq\{1,\ldots,n\}$. We identify a hypergraph with its edge set, and hence write $|H|$ for the number of edges of $H$. If $\mathcal{G}$ is a hypergraph, then $V(\mathcal{G})$ denotes its vertex set. Given an $r$-graph $\mathcal{G}$ and vertices $x,y\in V(\mathcal{G})$, the \emph{codegree} $\codeg_{\mathcal{G}}(x,y)$ is the number of edges of $\mathcal{G}$ containing both $x$ and $y$.

For positive integers $n,q,r$, let $T_r(n,q)$ denote the \emph{complete balanced $q$-partite $r$-graph} on $n$ vertices. Thus the vertex set is partitioned into $q$ classes, no two class sizes differ by more than one, and the edges are all $r$-sets meeting each class in at most one vertex. Let $\tr_r(n,q)\coloneqq|T_r(n,q)|$.
Equivalently, if the balanced class sizes are $n_1,\ldots,n_q$, then $\tr_r(n,q)\coloneqq\sum_{S\in\binom{[q]}r}\prod_{i\in S}n_i$, with the convention that $\tr_r(n,q)\coloneqq 0$ when $r>q$.

All polynomial rings below are over the same field $\mathbb{K}$. The characteristic-zero assumption is only needed for the differentiated diagonal-vanishing ideals; the monomial cover-ideal arguments are characteristic-free.

We use the \emph{standard grading} on every polynomial ring appearing below. Thus if $R=\mathbb{K}[z_1,\ldots,z_N]$, then
\[
R_d\coloneqq
\operatorname{span}_{\mathbb{K}}\left\{z_1^{a_1}\cdots z_N^{a_N}\colon a_1+\cdots+a_N=d\right\}
\]
is the $\mathbb{K}$-vector space of homogeneous polynomials of degree $d$.

An ideal $J$ is \emph{homogeneous} if it is generated by homogeneous polynomials, and it is a \emph{monomial ideal} if it is generated by monomials. A monomial is \emph{squarefree} if no variable appears with exponent larger than one. For a homogeneous ideal $J$, each polynomial in $J$ decomposes into homogeneous pieces which still belong to $J$. Hence the quotient $R/J$ inherits the grading of $R$: its \emph{degree-$d$ piece} is $(R/J)_d\coloneqq R_d/(J\cap R_d)$.
If $A=R/J$, we also write $A_d$ for this vector space; in particular, $A_{q+1}=0$ means that every homogeneous element of degree $q+1$ is zero in the quotient. When $J$ is a monomial ideal, the residue classes of the monomials not belonging to $J$ form a $\mathbb{K}$-basis of $R/J$; we call them \emph{standard monomials}.
Thus, for the monomial quotients used below, quotienting has a simple combinatorial meaning: the monomials in the ideal are declared to be zero, and the standard monomials are the monomials that survive.
When the variables are indexed by edges or vertices, these surviving monomials are exactly the subsets that avoid the forbidden products imposed by the ideal.

For a nonzero homogeneous ideal $J$ in a graded polynomial ring, define its \emph{initial degree} by $\alpha(J)\coloneqq\min\left\{d\colon J_d\ne 0\right\}$.
For a monomial ideal, this is the minimum degree of a monomial in $J$.

We also recall the squarefree monomial duality used later. If $R=\mathbb{K}[y_1,\ldots,y_N]$ and a squarefree monomial is written as $y_A\coloneqq\prod_{i\in A}y_i$, then the \emph{Alexander dual} of a squarefree monomial ideal $I=\langle y_A\colon A\in\mathcal A\rangle$ is
\[
I^\vee\coloneqq\bigcap_{A\in\mathcal A}\langle y_i\colon i\in A\rangle.
\]
In this paper, these dual ideals are the \emph{cover ideals}: a monomial lies in $I^\vee$ exactly when its support meets every set $A\in\mathcal A$.
Equivalently, if the sets $A$ are viewed as the edges of an auxiliary hypergraph on the variables, then $I^\vee$ is the monomial vertex-cover ideal of that auxiliary hypergraph.
Its squarefree monomials are precisely the sets of variables that hit every one of the sets $A$.

We next recall the clique-expansion family that appears in Mubayi's extremal formula.

Fix $\ell,r\ge 2$. Let $\mathcal{K}_{\ell}^{(r)}$ be the family of all simple $r$-graphs with at most $\binom{\ell}{2}$ edges which contain a set $S$ of $\ell$ vertices, called the \emph{core}, such that every pair of vertices in $S$ is contained in at least one edge.
Equivalently, one may choose an $r$-edge for each pair of core vertices, allowing the same edge to serve several pairs, and then take the resulting simple edge set.

When $r=2$, this family consists of the complete graph $K_\ell$. For $r>2$, the same definition encodes a positive codegree condition on all pairs in the core.

\section{Counterexamples to the polynomial-ideal conjecture}\label{sec:counterexamples}

We now prove \cref{thm:intro-counterexample}. The first step is the easy containment: generators coming from $(\ell-1)$-partite 3-graphs already satisfy the differentiated diagonal-vanishing conditions. The counterexamples below show that the reverse containment fails.

\begin{proof}[Proof of \cref{thm:intro-counterexample}]
We first record the easy containment as a claim.

\begin{claim}\label{clm:partite-generators-in-DI}
For all $\ell\ge 3$ and $n\ge \ell$, $\widehat{\mathrm{P}}^{(3)}(n,\ell)\subseteq \mathrm{DI}(n,\ell)$.
\end{claim}

\begin{proof}[Proof of the claim]
Since $\mathrm{DI}(n,\ell)$ is an ideal, it is enough to show that each generator $p_{\mathcal{G}}$ belongs to $\mathrm{DI}(n,\ell)$, where $\mathcal{G}$ is an $(\ell-1)$-partite 3-graph on $[n]$. Fix an $\ell$-set $L\subseteq[n]$. By the pigeonhole principle, two vertices $a,b\in L$ lie in the same part of the $(\ell-1)$-partition. No edge of $\mathcal{G}$ contains both $a$ and $b$. Hence every triple $\{a,b,c\}$ with $c\in[n]\setminus\{a,b\}$ is missing from $\mathcal{G}$, and the factor $(x_a-x_b)$ occurs once for each such $c$. Thus $(x_a-x_b)^{n-2}\mid p_{\mathcal{G}}$.
Now fix $i\in[n]$ and $0\le j\le n-3$. If $i\notin\{a,b\}$, this factor is independent of $x_i$ and divides $\partial^j p_{\mathcal{G}}/\partial x_i^j$. If $i=a$ or $i=b$, differentiating $j$ times can reduce the exponent of $(x_a-x_b)$ by at most $j$, so every term of $\partial^j p_{\mathcal{G}}/\partial x_i^j$ remains divisible by $(x_a-x_b)^{n-2-j}$, which has positive exponent. Therefore
$\partial^j p_{\mathcal{G}}/\partial x_i^j$ vanishes when the variables indexed by $L$ are identified. Since $L$, $i$, and $j$ were arbitrary, $p_{\mathcal{G}}\in \mathrm{DI}(n,\ell)$.
\end{proof}

By the claim, it remains only to exhibit an element of $\mathrm{DI}(n,\ell)$ which is not in $\widehat{\mathrm{P}}^{(3)}(n,\ell)$. We split the construction into three cases.

\medskip\noindent\textit{Case 1: $\ell\ge 4$.}
Put $q\coloneqq\ell-1$, so $q\ge 3$, and choose a balanced partition
$[n]=V_1\sqcup\cdots\sqcup V_{\ell-2}$ into $\ell-2=q-1$ nonempty parts. Define the Vandermonde-type product
\[
F_{n,\ell}\coloneqq\prod_{s=1}^{\ell-2}\prod_{\substack{a<b\\ a,b\in V_s}}(x_a-x_b).
\]
We first verify that $F_{n,\ell}\in \mathrm{DI}(n,\ell)$. Fix an $\ell$-set $L\subseteq[n]$ and a variable $x_i$. If $i\in L$, then the $\ell-1$ vertices of $L\setminus\{i\}$ are distributed among only $\ell-2$ parts, and hence two of them, say $a$ and $b$, lie in the same part. The factor $x_a-x_b$ divides $F_{n,\ell}$ and is independent of $x_i$. Therefore it also divides every derivative $\partial^jF_{n,\ell}/\partial x_i^j$.

If $i\notin L$, then the $\ell$ vertices of $L$ themselves are distributed among $\ell-2$ parts, so the same argument gives a pair $a,b\in L$ with $x_a-x_b$ dividing every derivative with respect to $x_i$. In both cases the derivative vanishes on the diagonal indexed by $L$. Since $L,i$, and $j$ were arbitrary, $F_{n,\ell}\in \mathrm{DI}(n,\ell)$.

It remains to show that $F_{n,\ell}\notin \widehat{\mathrm{P}}^{(3)}(n,\ell)$. Each generator $p_{\mathcal{G}}$ of $\widehat{\mathrm{P}}^{(3)}(n,\ell)$ is homogeneous of degree $\deg p_{\mathcal{G}}=3\left(\binom n3-|\mathcal{G}|\right)$. The largest number of edges in an $(\ell-1)$-partite 3-graph on $[n]$ is $\tr_3(n,\ell-1)=\tr_3(n,q)$. Hence all homogeneous generators of $\widehat{\mathrm{P}}^{(3)}(n,\ell)$ have degree at least $D_{n,\ell}\coloneqq 3\left(\binom n3-\tr_3(n,q)\right)$. Since $\widehat{\mathrm{P}}^{(3)}(n,\ell)$ is homogeneous and generated in degrees at least $D_{n,\ell}$, it contains no nonzero homogeneous element of smaller degree.

We now compare degrees. The chosen partition has $q-1$ balanced classes, and therefore
\[
\deg F_{n,\ell}=\sum_{s\in[q-1]}\binom{|V_s|}{2}<\frac{n^2}{2(q-1)}.
\]
Let $W_1,\ldots,W_q$ be the balanced $q$-partition defining $T_3(n,q)$, and set $a\coloneqq\left\lceil n/q\right\rceil$.
The triples missing from $T_3(n,q)$ include all triples having two vertices in one part and the third vertex outside that part. Thus
\[
\binom n3-\tr_3(n,q)\ge \sum_{s\in[q]}\binom{|W_s|}{2}(n-|W_s|).
\]
These counted triples are disjoint, since a triple with two vertices in one part and the third vertex outside that part determines the repeated part uniquely. Since $|W_s|\le a$ for every $s$ and $n\ge q+1$, and since $\binom m2\ge m-1$ for every integer $m\ge 0$, we have $\sum_{s\in[q]}\binom{|W_s|}{2}\ge n-q$. Consequently $\binom n3-\tr_3(n,q)\ge (n-q)(n-a)$. It follows that $D_{n,\ell}\ge 3(n-q)(n-a)$. We claim that $3(n-q)(n-a)>n^2/(2(q-1))$. Indeed, $a\le (n+q-1)/q$, whence $n-a\ge (q-1)(n-1)/q$.
Moreover, the function $(n-q)(n-1)/n^2$ is increasing for $n\ge q+1$ and equals $q/(q+1)^2$ at $n=q+1$. Therefore
\[
\frac{6(q-1)(n-q)(n-a)}{n^2}
\ge
\frac{6(q-1)^2}{q}\cdot \frac{(n-q)(n-1)}{n^2}
\ge
\frac{6(q-1)^2}{(q+1)^2}>1
\]
for every $q\ge 3$. Hence $\deg F_{n,\ell}<D_{n,\ell}$, which proves $F_{n,\ell}\notin \widehat{\mathrm{P}}^{(3)}(n,\ell)$.

\medskip\noindent\textit{Case 2: $\ell=3$ and $n\ge 4$.}
Take the full Vandermonde product $F_{n,3}\coloneqq\prod_{1\le a<b\le n}(x_a-x_b)$.
For any 3-set $L$ and any $i\in[n]$, if $i\in L$ then the two vertices of $L\setminus\{i\}$ give a factor independent of $x_i$; if $i\notin L$, any two vertices of $L$ give such a factor. Thus every derivative with respect to $x_i$ that appears in the definition of $\mathrm{DI}(n,3)$ vanishes on the diagonal indexed by $L$. Hence $F_{n,3}\in \mathrm{DI}(n,3)$.

A 2-partite 3-graph has no edges. Therefore every generator of $\widehat{\mathrm{P}}^{(3)}(n,3)$ has degree $3\binom n3$. On the other hand, $\deg F_{n,3}=\binom n2<3\binom n3$ for $n\ge 4$. Hence $F_{n,3}\notin \widehat{\mathrm{P}}^{(3)}(n,3)$.

\medskip\noindent\textit{Case 3: $(\ell,n)=(3,3)$.}
Here $n-3=0$, so $\mathrm{DI}(3,3)=\mathrm{I}(3,3)$. The linear polynomial $F_{3,3}\coloneqq x_1-x_2$ vanishes on the diagonal $x_1=x_2=x_3$, and therefore belongs to $\mathrm{DI}(3,3)$. However, every 2-partite 3-graph on $[3]$ has no edge, so $\widehat{\mathrm{P}}^{(3)}(3,3)$ is generated by $(x_1-x_2)(x_1-x_3)(x_2-x_3)$. This ideal contains no nonzero element of degree one. Hence $F_{3,3}\notin \widehat{\mathrm{P}}^{(3)}(3,3)$.

The three cases prove \cref{thm:intro-counterexample}.
\end{proof}

\section{The cover-ideal method}\label{sec:cover-ideal-method}

This section develops the monomial-ideal framework used in the second part of the paper. The first subsection proves the Hilbert-function symmetrization result that replaces clique-counting symmetrization. The second subsection translates ordinary and generalized Tur\'{a}n problems into cover-ideal language. The final subsection applies this translation to Mubayi's clique-expansion family, where the cover ideal becomes a missing star ideal.

\subsection{Square-zero quotients and Hilbert-function symmetrization}\label{sec:hilbert-symmetrization}

This subsection proves the algebraic extremal lemma used later in place of an auxiliary-graph clique-counting theorem. It will be applied twice: first to the clique-counting generalized Tur\'{a}n problem, and then to the initial-degree computation for the star ideal. The objects are quadratic monomial quotients with all variables square-zero. Their Hilbert functions give the same numbers, but the proof below is written entirely in terms of standard monomials, links, quotients, and cloning of variables.

Let $S\coloneqq\mathbb{K}[x_1,\ldots,x_n]$, and let $A\coloneqq S/I$, where $I$ is generated by $x_1^2,\ldots,x_n^2$ together with some squarefree quadratic monomials $x_ax_b$.
We call such an $A$ a \emph{square-zero quadratic monomial quotient}.
Equivalently, this quotient can be read as a graph on vertex set $[n]$: join $i$ and $j$ when $x_ix_j$ is nonzero in $A$.
Then the standard degree-$d$ monomials are exactly the $d$-cliques of this graph, and the condition $A_{q+1}=0$ says that there is no $(q+1)$-clique.
The theorem below is therefore a Zykov-type clique-counting statement written in quotient language.

For fixed $q\ge 1$ and $r\ge 1$, define $\mathcal H(n,q,r)\coloneqq\max_A \dim_{\mathbb{K}} A_r$, where the maximum is over all square-zero quadratic monomial quotients $A$ on $n$ variables with $A_{q+1}=0$.
The main result of this subsection is $\mathcal H(n,q,r)=\tr_r(n,q)$. The proof below gives the upper bound, which is the only direction needed later. The reverse inequality is attained by the complete balanced $q$-partite square-zero quotient: take a balanced partition of $[n]$ into $q$ parts and kill exactly the squares and the products of variables in the same part. Its degree-$r$ Hilbert function is $\tr_r(n,q)$ and its degree-$(q+1)$ piece is zero.

Two variables $x_i,x_j$ are called \emph{parallel} in $A$ if $x_ix_j=0$ and, for every $h\notin\{i,j\}$, $x_ix_h=0$ if and only if $x_jx_h=0$. This is an equivalence relation on the variables. Its equivalence classes will be called \emph{parallel classes}. Inside each parallel class all pairwise products vanish; between two distinct classes either all products vanish or all products are nonzero.
In the graph interpretation above, variables in the same parallel class are twin vertices: they are nonadjacent to each other and have the same neighbors outside the class.

If $C$ is a parallel class and $x_c\in C$, set $\lambda_C(d)\coloneqq\dim_{\mathbb{K}}\operatorname{span}\left\{w\in A_d\colon x_cw\ne 0\text{ in }A\right\}$.
Equivalently, $\lambda_C(d)=\dim_{\mathbb{K}}(x_cA_d)$. This number is independent of the choice of $x_c\in C$.

The proof of the upper bound uses one local symmetrization move: when two parallel classes have zero product between them, we clone the class which gives the larger degree-$r$ contribution.
Iterating this move gives the Hilbert-function form of the Tur\'{a}n bound needed later.

\begin{theorem}\label{thm:hilbert-turan}
Let $A\coloneqq\mathbb{K}[x_1,\ldots,x_n]/I$ be a square-zero quadratic monomial quotient. If $A_{q+1}=0$, then, for every $r\ge 1$, $\dim_{\mathbb{K}} A_r\le \tr_r(n,q)$. Here $\tr_r(n,q)$ is the degree-$r$ Hilbert function of the complete balanced $q$-partite square-zero quotient. The bound is sharp.
\end{theorem}

This theorem is best viewed as the usual Zykov symmetrization~\cite{Zykov} written in quotient language. Indeed, if one applies it to the Stanley--Reisner quotient of a graph $F$, then $\dim_{\mathbb{K}} A_r$ counts the $r$-cliques of $F$, while $A_{q+1}=0$ says that $F$ has no $(q+1)$-clique. Its limitation is also worth noting: the result is not a new strengthening of Zykov's theorem, and the main idea behind its proof is still essentially combinatorial. The point of the formulation is that the same balancing argument can be carried out directly on square-zero quadratic monomial quotients, which is the form needed for the cover-ideal applications below.

\begin{proof}
The case $r=1$ is immediate, so fix $r\ge 2$. We first prove the local cloning step.

\begin{claim}\label{clm:cloning}
Let $A$ be a square-zero quadratic monomial quotient with $A_{q+1}=0$. Let $U$ and $V$ be distinct parallel classes such that $x_ux_v=0$ for all $u\in U$ and $v\in V$. Then one can replace either $U$ by clones of $V$ or $V$ by clones of $U$ so as to obtain another square-zero quadratic monomial quotient $A'$ on the same $n$ variables with $(A')_{q+1}=0$ and $\dim_{\mathbb{K}} (A')_r\ge \dim_{\mathbb{K}} A_r$.
\end{claim}

\begin{proof}[Proof of the claim]
We describe the operation $V\leftarrow U$. Let $A^{V\leftarrow U}$ be the quotient obtained as follows. All products inside $U\cup V$ are set equal to zero. For $v\in V$ and for a variable $z\notin U\cup V$, we impose $x_vz=0$ if and only if $x_uz=0$ for some (equivalently, every) $u\in U$. All products not involving $V$ are kept as in $A$. Thus the variables in $V$ become clones of the variables in $U$.

First, $A^{V\leftarrow U}_{q+1}=0$. Indeed, a standard monomial of degree $q+1$ in $A^{V\leftarrow U}$ that avoids $V$ would already have been standard in $A$, impossible. A standard monomial that contains a variable $x_v$ with $v\in V$ contains no other variable from $U\cup V$; after replacing $x_v$ by any $x_u$ with $u\in U$, its support would be standard in $A$ by construction. This would give a nonzero element of $A_{q+1}$, again impossible.

Now compare degree $r$. The standard monomials of degree $r$ avoiding $V$ are unchanged. Since products inside $V$ and between $U$ and $V$ are zero before and after cloning, a standard degree-$r$ monomial containing a variable of $V$ contains exactly one such variable and no variable from $U$. In $A$ these monomials are counted by $|V|\lambda_V(r-1)$, whereas in $A^{V\leftarrow U}$ they are counted by $|V|\lambda_U(r-1)$.
Consequently
\[
\dim_{\mathbb{K}} A^{V\leftarrow U}_r-\dim_{\mathbb{K}} A_r
=|V|\bigl(\lambda_U(r-1)-\lambda_V(r-1)\bigr).
\]
If this quantity is nonnegative, take $A'\coloneqq A^{V\leftarrow U}$. Otherwise perform the opposite cloning $U\leftarrow V$, for which the corresponding difference is $|U|\bigl(\lambda_V(r-1)-\lambda_U(r-1)\bigr)>0$.
In either case the degree-$r$ Hilbert function does not decrease, and the classes $U,V$ are merged.
\end{proof}

Starting from $A$, repeatedly apply the claim whenever two distinct parallel classes have zero product between them. Each step preserves the condition that the degree-$(q+1)$ piece is zero and does not decrease the degree-$r$ Hilbert function. It also decreases the number of parallel classes. Hence the process terminates after finitely many steps.

Let $B$ be the terminal quotient. Then $\dim_{\mathbb{K}} A_r\le \dim_{\mathbb{K}} B_r$ and $B_{q+1}=0$, and any two distinct parallel classes of $B$ have nonzero product between them. Thus there is a partition $[n]=P_1\sqcup\cdots\sqcup P_s$ such that $x_ix_j=0$ if and only if $i,j$ belong to the same $P_a$. Writing $n_a\coloneqq|P_a|$, the standard degree-$r$ monomials of $B$ are obtained by choosing $r$ distinct classes and one variable from each chosen class. Hence $\dim_{\mathbb{K}} B_r=e_r(n_1,\ldots,n_s)$, where $e_r$ denotes the elementary symmetric polynomial of degree $r$.

Since $B_{q+1}=0$, one must have $s\le q$; otherwise a product of one variable from each of $q+1$ distinct classes would be nonzero. By appending zero parts, we may regard $(n_1,\ldots,n_s)$ as a $q$-tuple summing to $n$.

It remains to show that $e_r(n_1,
\ldots,n_q)$ is maximized by a balanced $q$-tuple. Suppose two entries satisfy $a\ge b+2$, and let $\mathbf c$ denote the remaining entries. Replacing $(a,b)$ by $(a-1,b+1)$ changes the value of $e_r$ by
\[
e_r(\mathbf c,a-1,b+1)-e_r(\mathbf c,a,b)
=(a-b-1)e_{r-2}(\mathbf c)\ge 0,
\]
with the convention $e_0=1$ and $e_d=0$ for $d<0$. Repeated balancing therefore leads to a balanced $q$-tuple without decreasing $e_r$. Thus $\dim_{\mathbb{K}} B_r\le \tr_r(n,q)$, and the theorem follows.
\end{proof}

\subsection{Cover ideals for ordinary and generalized Tur\'{a}n problems}\label{sec:cover-ideals}

We now give the general algebraic translation between forbidden-copy conditions and cover ideals. Let $\mathcal F$ be a family of $r$-graphs, and let $R_{n,r}\coloneqq\mathbb{K}[y_E\colon E\in\binom{[n]}r]$ be the polynomial ring with one variable for each possible $r$-edge. Denote by $\mathcal C_{\mathcal F,n}$ the set of all copies, inside $\binom{[n]}r$, of members of $\mathcal F$.

The $\mathcal F$-free edge sets form a \emph{simplicial complex}
\[
\Delta_{\mathcal F,n}\coloneqq
\left\{\mathcal{G}\subseteq \tbinom{[n]}r\colon
\mathcal{G}\text{ is }\mathcal F\text{-free}\right\}.
\]
Its \emph{Stanley--Reisner ideal} is
\[
I_{\mathcal F,n}\coloneqq\left\langle\prod_{E\in H}y_E\colon H\in \mathcal C_{\mathcal F,n}\right\rangle.
\]
The \emph{Alexander-dual cover ideal} is
\[
C_{\mathcal F,n}\coloneqq I_{\mathcal F,n}^{\vee}
=\bigcap_{H\in \mathcal C_{\mathcal F,n}}\langle y_E\colon E\in H\rangle.
\]
Thus a monomial lies in $C_{\mathcal F,n}$ precisely when its support meets every forbidden copy.
In concrete terms, a squarefree monomial $y_M$ records a set $M$ of $r$-edges that are missing.
The condition $y_M\in C_{\mathcal F,n}$ says that every forbidden copy contains at least one edge from $M$, or equivalently that the complementary $r$-graph $\binom{[n]}r\setminus M$ is $\mathcal F$-free.

The next lemma makes the translation to the ordinary Tur\'{a}n problem exact.

\begin{lemma}\label{lem:cover-translation}
For every family $\mathcal F$ of $r$-graphs, $\alpha(C_{\mathcal F,n})=\binom nr-\ex(n,\mathcal F)$. Moreover, the squarefree monomials of degree $\alpha(C_{\mathcal F,n})$ are exactly the complements of extremal $\mathcal F$-free $r$-graphs.
\end{lemma}

\begin{proof}
Since $C_{\mathcal F,n}$ is a squarefree monomial ideal, its initial degree is attained by a squarefree monomial. For $M\subseteq \binom{[n]}r$, write $y_M\coloneqq\prod_{E\in M}y_E$.
Then $y_M\in C_{\mathcal F,n}$ if and only if $M$ intersects every forbidden copy $H\in\mathcal C_{\mathcal F,n}$. This is equivalent to saying that the complementary edge set
$\binom{[n]}r\setminus M$ is $\mathcal F$-free. Minimizing $|M|$ over all squarefree monomials in $C_{\mathcal F,n}$ is therefore the same as maximizing the number of edges in an $\mathcal F$-free $r$-graph. Hence $\alpha(C_{\mathcal F,n})=\binom nr-\ex(n,\mathcal F)$. The characterization of the monomials attaining the initial degree follows from the same equivalence.
\end{proof}

The limitation of this translation is worth making explicit.
For a general forbidden family, computing the initial degree of $C_{\mathcal F,n}$ is essentially the original extremal problem. The useful point is that some families, including $\mathcal{K}_{\ell}^{(r)}$, give cover ideals with additional structure.

We next adapt this translation to generalized Tur\'{a}n numbers, where the forbidden graph and the graph being counted may be different.

Let $T$ be the graph to be counted and $F$ the graph to be forbidden. We assume, as usual, that both graphs have no isolated vertices. Let $R_n^{(2)}\coloneqq\mathbb{K}[y_e\colon e\in K_n]$ be the edge-variable ring of the complete graph. Let $\mathcal T_{T,n}$ be the set of all copies of $T$ in $K_n$, and let $\mathcal C_{F,n}$ be the set of all copies of $F$ in $K_n$. The \emph{forbidden-copy ideal} and its Alexander dual are
\[
I_{F,n}\coloneqq\langle y_A\colon A\in \mathcal C_{F,n}\rangle
\quad\text{and}\quad 
C_{F,n}\coloneqq I_{F,n}^{\vee}=\bigcap_{A\in \mathcal C_{F,n}}\langle y_e\colon e\in A\rangle,
\]
where $y_A\coloneqq\prod_{e\in A}y_e$.

The copies of $T$ are represented by the finite-dimensional \emph{monomial subspace} $V_T(n)\subseteq R_n^{(2)}$, namely the $\mathbb{K}$-span of the monomials $y_B$ with $B\in\mathcal T_{T,n}$.
For $M\subseteq K_n$, set $\mathfrak m_M\coloneqq\langle y_e\colon e\in M\rangle$ and $Q_M\coloneqq R_n^{(2)}/\mathfrak m_M$, and let $\pi_M\colon R_n^{(2)}\longrightarrow Q_M$ be the quotient map.
Here passing to $Q_M$ simply sets the edge variables $y_e$ with $e\in M$ equal to zero.
A monomial for a copy of $T$ survives in $Q_M$ exactly when that copy uses no edge from $M$.
The \emph{quotient rank} of $M$ on the $T$-copy space is $\rho_T(M)\coloneqq\dim_{\mathbb{K}} \pi_M(V_T(n))$. This is the number of $T$-copy monomials which remain nonzero after the variables in $M$ are killed. Dually, define the \emph{$T$-initial loss} of the cover ideal by
\[
\alpha_T(C_{F,n})\coloneqq
\min_{\substack{y_M\in C_{F,n}\\ y_M\text{ squarefree}}}
\dim_{\mathbb{K}} \ker\bigl(\pi_M|_{V_T(n)}\bigr)
=
\min_{\substack{y_M\in C_{F,n}\\ y_M\text{ squarefree}}}
\dim_{\mathbb{K}}\bigl(V_T(n)\cap\mathfrak m_M\bigr).
\]
Thus $\alpha_T$ is defined by a quotient map rather than by ordinary degree. Since $V_T(n)$ has the monomial basis indexed by $T$-copies, this loss is equivalently the number of $T$-copy monomials killed by the missing-edge ideal $\mathfrak m_M$.

When $T=K_2$, the space $V_T(n)$ is spanned by the edge variables, so $\alpha_{K_2}(C_{F,n})=\alpha(C_{F,n})$ and the construction specializes to the ordinary cover-ideal translation in \cref{lem:cover-translation}.
For general $T$, ordinary degree no longer measures the objective; after choosing an extremal missing-edge monomial $y_{M_0}$, the task is to compare the quotient ranks $\dim_{\mathbb{K}}\pi_M(V_T(n))$ for squarefree monomials $y_M\in C_{F,n}$.
In the clique-counting case $T=K_s$ and $F=K_{q+1}$, this comparison becomes the Hilbert-function inequality of \cref{thm:hilbert-turan}.

With the quotient-rank notation in place, the generalized Tur\'{a}n number is obtained by maximizing the surviving $T$-copy space.

\begin{lemma}\label{lem:generalized-cover}
For fixed graphs $T$ and $F$ without isolated vertices, $\ex(n,T,F)=|\mathcal T_{T,n}|-\alpha_T(C_{F,n})$.
Equivalently,
\[
\ex(n,T,F)=\max_{\substack{y_M\in C_{F,n}\\ y_M\text{ squarefree}}}\dim_{\mathbb{K}} \pi_M(V_T(n)).
\]
Moreover, for $2\le s\le q+1$, $\ex(n,K_s,K_{q+1})=\tr_s(n,q)$, where $\tr_s(n,q)$ is the number of $s$-cliques in $T_2(n,q)$.
\end{lemma}

\begin{proof}
The squarefree monomials in $C_{F,n}$ are exactly the monomial vertex covers of the family of forbidden-copy monomials $\left\{y_A\colon A\in\mathcal C_{F,n}\right\}$. Equivalently, $y_M\in C_{F,n}$ if and only if every forbidden-copy monomial $y_A$ is killed in the quotient $Q_M=R_n^{(2)}/\mathfrak m_M$.

For such an $M$, the quotient $Q_M$ has the standard monomial basis consisting of monomials in variables $y_e$ with $e\notin M$. Hence a basis element $y_B\in V_T(n)$ has nonzero image under $\pi_M$ precisely when none of its variables is killed, i.e.\ precisely when $y_B\notin \mathfrak m_M$. Distinct nonzero images remain distinct standard monomials in $Q_M$. Therefore $\rho_T(M)=\dim_{\mathbb{K}}\pi_M(V_T(n))$ is exactly the number of $T$-copy monomials surviving in the quotient.

Thus maximizing the number of $T$-copies in an $F$-free graph is the same algebraic problem as maximizing $\rho_T(M)$ over the squarefree monomials $y_M\in C_{F,n}$. Since $\dim_{\mathbb{K}} V_T(n)=|\mathcal T_{T,n}|$ and $\dim_{\mathbb{K}} \pi_M(V_T(n))=\dim_{\mathbb{K}} V_T(n)-\dim_{\mathbb{K}}\ker(\pi_M|_{V_T(n)})$,
the maximum quotient rank equals
\[
|\mathcal T_{T,n}|-
\min_{\substack{y_M\in C_{F,n}\\ y_M\text{ squarefree}}}
\dim_{\mathbb{K}}\ker(\pi_M|_{V_T(n)}).
\]
This is the asserted identity.

For the moreover statement, we recover the clique-counting form of Zykov's theorem.
Fix a squarefree monomial $y_M\in C_{K_{q+1},n}$.
We associate to $M$ the vertex-variable quotient $A_M\coloneqq\mathbb{K}[z_1,\ldots,z_n]/\langle z_i^2,\ z_iz_j\colon y_{ij}\in M\rangle$.
Because $y_M$ lies in the cover ideal of the $K_{q+1}$-copy ideal, every $(q+1)$-vertex set $L\subseteq[n]$ contains a pair $ij\subseteq L$ with $y_{ij}\in M$. Hence $\prod_{i\in L}z_i=0$ in $A_M$ for every $|L|=q+1$, and therefore $(A_M)_{q+1}=0$. By \cref{thm:hilbert-turan}, $\dim_{\mathbb{K}}(A_M)_s\le \tr_s(n,q)$.

It remains to identify this Hilbert function with the quotient rank on the $K_s$-copy space. The linear map
\[
\Phi_s\colon V_{K_s}(n)\longrightarrow \mathbb{K}[z_1,
\ldots,z_n]_s,
\qquad
\Phi_s\bigl(y_{K_W}\bigr)\coloneqq\prod_{i\in W}z_i
\]
identifies the monomial basis indexed by $s$-vertex complete subgraphs with the squarefree degree-$s$ monomials in the vertex variables. After quotienting by $\mathfrak m_M$ on the source and by the quadratic relations defining $A_M$ on the target, the same basis elements are killed: $y_{K_W}$ is killed in $R_n^{(2)}/\mathfrak m_M$ if and only if some $y_{ij}\in M$ with $i,j\in W$, and this is equivalent to $\prod_{i\in W}z_i=0$ in $A_M$. Thus $\dim_{\mathbb{K}}\pi_M(V_{K_s}(n))=\dim_{\mathbb{K}}(A_M)_s\le \tr_s(n,q)$. \cref{lem:generalized-cover} gives $\ex(n,K_s,K_{q+1})\le \tr_s(n,q)$.

For equality, take a balanced partition of $[n]$ into $q$ parts and let $M_0$ be the set of all pairs contained in a single part. Then $y_{M_0}\in C_{K_{q+1},n}$, and the quotient $A_{M_0}$ is the complete balanced $q$-partite square-zero quotient. Its degree-$s$ Hilbert function is $\tr_s(n,q)$, so the upper bound is attained.
\end{proof}

\subsection{The star ideal}\label{sec:star-ideal}

We now apply the cover-ideal translation to the family $\mathcal{K}_{\ell}^{(r)}$. For an $r$-graph $\mathcal{G}\subseteq\binom{[n]}r$, define its \emph{missing-edge monomial} by $m_{\mathcal{G}}\coloneqq\prod_{E\notin \mathcal{G}}y_E$.
Thus $\deg m_{\mathcal{G}}=\binom nr-|\mathcal{G}|$.

For a pair $\{a,b\}\subseteq[n]$, let $\mathcal S^{(r)}_{ab}\coloneqq\left\{E\in\binom{[n]}r\colon\{a,b\}\subseteq E\right\}$ be the \emph{$r$-uniform star} of the pair $\{a,b\}$. Define the \emph{missing star monomial}
\[
u^{(r)}_{ab}\coloneqq\prod_{E\in \mathcal S^{(r)}_{ab}}y_E.
\]
Then $u^{(r)}_{ab}\mid m_{\mathcal{G}}$ if and only if $\codeg_{\mathcal{G}}(a,b)=0$.
Indeed, $u^{(r)}_{ab}$ contains one variable for every possible $r$-edge through the pair $\{a,b\}$, so divisibility by $u^{(r)}_{ab}$ means that all such edges are missing from $\mathcal{G}$.

The next definition imposes this missing-star condition on every possible $\ell$-vertex core.

For $\ell,r\ge 2$, define
\[
J^{(r)}_{n,\ell}\coloneqq\bigcap_{L\in\binom{[n]}{\ell}}
\bigl\langle u^{(r)}_{ab}\colon\{a,b\}\subseteq L\bigr\rangle
\subseteq R_{n,r}.
\]
For a missing-edge monomial $m_{\mathcal{G}}$, membership in $J^{(r)}_{n,\ell}$ therefore says that every $\ell$-set $L$ contains a pair $\{a,b\}$ whose whole $r$-uniform star is missing from $\mathcal{G}$.
This is exactly the local obstruction to using $L$ as the core of a member of $\mathcal{K}_{\ell}^{(r)}$.

The useful point is that, for $\mathcal{K}_{\ell}^{(r)}$, this local missing-star ideal is exactly the cover ideal.

\begin{proposition}\label{prop:collapse}
For $\mathcal{K}_{\ell}^{(r)}$, we have $C_{\mathcal{K}_{\ell}^{(r)},n}=J^{(r)}_{n,\ell}$. Consequently, for every $r$-graph $\mathcal{G}\subseteq\binom{[n]}r$, we have $m_{\mathcal{G}}\in J^{(r)}_{n,\ell}$ if and only if $\mathcal{G}$ is $\mathcal{K}_{\ell}^{(r)}$-free.
\end{proposition}

\begin{proof}
Fix an $\ell$-set $L\subseteq[n]$. Let
\[
C_L\coloneqq
\bigcap_{\substack{H\in \mathcal{K}_{\ell}^{(r)}\\ H\text{ has core }L}}
\langle y_E\colon E\in H\rangle
\]
be the part of the cover ideal arising from forbidden copies with core $L$. We claim that $C_L=\bigl\langle u^{(r)}_{ab}\colon\{a,b\}\subseteq L\bigr\rangle$.
Since both sides are monomial ideals, it suffices to test membership of monomials.
Suppose first that a monomial $m$ is divisible by $u^{(r)}_{ab}$ for some pair $\{a,b\}\subseteq L$. Every $H\in \mathcal{K}_{\ell}^{(r)}$ with core $L$ contains an edge with the pair $\{a,b\}$. The variable corresponding to that edge divides $u^{(r)}_{ab}$, and hence divides $m$. Thus $m\in\langle y_E\colon E\in H\rangle$ for every such $H$, so $m\in C_L$.

Conversely, suppose that $m$ is not divisible by any $u^{(r)}_{ab}$ with $\{a,b\}\subseteq L$. For each pair $\{a,b\}\subseteq L$, choose an $r$-edge $E_{ab}$ containing $\{a,b\}$ such that $y_{E_{ab}}\nmid m$. The $r$-graph with edge set $\left\{E_{ab}\colon\{a,b\}\subseteq L\right\}$ has core $L$ and belongs to $\mathcal{K}_{\ell}^{(r)}$; repeated choices only reduce the number of distinct edges. None of its edge variables divides $m$, so $m$ does not lie in the ideal generated by the variables of this forbidden copy. Hence $m\notin C_L$.

The equality for each fixed core $L$ follows. Intersecting over all $L\in\binom{[n]}{\ell}$ gives $C_{\mathcal{K}_{\ell}^{(r)},n}=J^{(r)}_{n,\ell}$. The final equivalence is the corresponding cover-ideal interpretation for the missing-edge monomial $m_{\mathcal{G}}$.
\end{proof}

After this identification, the remaining task is to compute the initial degree of the missing star ideal.

\begin{theorem}\label{thm:initial-degree}
For all $n,\ell,r\ge 2$, $\alpha\bigl(J^{(r)}_{n,\ell}\bigr)=\binom nr-\tr_r(n,\ell-1)$. Consequently, $\ex(n,\mathcal{K}_{\ell}^{(r)})=\tr_r(n,\ell-1)$.
\end{theorem}

\begin{proof}
If $n<\ell$, then the intersection defining $J^{(r)}_{n,\ell}$ is empty, so $J^{(r)}_{n,\ell}=R_{n,r}$. Also $\tr_r(n,\ell-1)=\binom nr$, since the balanced $(\ell-1)$-partite $r$-graph can place the $n$ vertices in distinct parts. Thus both sides are zero. Hence assume $n\ge \ell$.

Put $q\coloneqq\ell-1$. Since $J^{(r)}_{n,\ell}$ is a squarefree monomial ideal, its initial degree is attained by a squarefree monomial. Let $m\in J^{(r)}_{n,\ell}$ be squarefree. Define the associated \emph{vertex-variable quotient}
\[
A_m\coloneqq
\mathbb{K}[x_1,\ldots,x_n]\Big/
\bigl\langle x_i^2,\ x_ax_b\text{ whenever }u^{(r)}_{ab}\mid m\bigr\rangle.
\]
This is a square-zero quadratic monomial quotient.
This quotient turns the missing-star information in $m$ into forbidden pairs among the vertex variables: the product $x_ax_b$ is killed exactly when $m$ contains all $r$-edge variables through $\{a,b\}$.
Thus standard degree-$r$ monomials in $A_m$ give $r$-sets whose every pair still has at least one corresponding $r$-edge variable not forced into $m$.

We first show that $(A_m)_\ell=0$. Let $L\in\binom{[n]}{\ell}$. Since $m\in \bigl\langle u^{(r)}_{ab}\colon\{a,b\}\subseteq L\bigr\rangle$, there exists a pair $\{a,b\}\subseteq L$ with $u^{(r)}_{ab}\mid m$. By the definition of $A_m$, this means $x_ax_b=0$ in $A_m$. Hence $\prod_{i\in L}x_i=0$ in $A_m$. Since this holds for every $\ell$-set $L$, we have $(A_m)_\ell=(A_m)_{q+1}=0$. By \cref{thm:hilbert-turan}, $\dim_{\mathbb{K}}(A_m)_r\le \tr_r(n,q)=\tr_r(n,\ell-1)$.

Now consider an $r$-set $E\in\binom{[n]}r$ such that $y_E\nmid m$. We claim that $x_E\coloneqq\prod_{i\in E}x_i$ is nonzero in $(A_m)_r$. If $x_E=0$, then, since the defining ideal of $A_m$ is generated by squares and squarefree quadratic monomials and $E$ is squarefree, there is a pair $\{a,b\}\subseteq E$ such that $x_ax_b=0$ in $A_m$. By construction of $A_m$, this is equivalent to $u^{(r)}_{ab}\mid m$. But $E$ contains $\{a,b\}$, so $y_E$ is one of the factors in $u^{(r)}_{ab}$. Hence $y_E\mid m$, a contradiction.

Thus every variable $y_E$ not dividing $m$ gives a distinct nonzero standard monomial $x_E\in(A_m)_r$. These monomials are linearly independent, so
\[
\#\left\{E\in\tbinom{[n]}r\colon y_E\nmid m\right\}
\le \dim_{\mathbb{K}}(A_m)_r
\le \tr_r(n,\ell-1).
\]
Because $m$ is squarefree, $\deg m=\binom nr-\#\left\{E\in\binom{[n]}r\colon y_E\nmid m\right\}$. Therefore $\deg m\ge \binom nr-\tr_r(n,\ell-1)$. Since $m\in J^{(r)}_{n,\ell}$ was arbitrary among squarefree monomials, this proves $\alpha(J^{(r)}_{n,\ell})\ge \binom nr-\tr_r(n,\ell-1)$.

For the reverse inequality, take a balanced $(\ell-1)$-partition $\mathcal V=(V_1,\ldots,V_{\ell-1})$ of $[n]$.
Let
\[
\mathcal N_{\mathcal V}^{(r)}
\coloneqq
\left\{E\in\tbinom{[n]}r\colon E\text{ contains two vertices from some }V_i\right\}
\]
be the set of \emph{non-transversal $r$-sets} with respect to $\mathcal V$, and set $m_0\coloneqq\prod_{E\in \mathcal N_{\mathcal V}^{(r)}}y_E$. The \emph{transversal $r$-sets} are exactly the edges of $T_r(n,\ell-1)$, and hence $\deg m_0=\binom nr-\tr_r(n,\ell-1)$.
We verify that $m_0\in J^{(r)}_{n,\ell}$. Let $L\in\binom{[n]}{\ell}$. Since $L$ has $\ell$ vertices and there are only $\ell-1$ parts, two vertices $a,b\in L$ lie in the same part. Every $r$-set $E$ containing $\{a,b\}$ is non-transversal, so $y_E\mid m_0$. Hence $u^{(r)}_{ab}\mid m_0$, and therefore $m_0\in \bigl\langle u^{(r)}_{ab}\colon\{a,b\}\subseteq L\bigr\rangle$. Since $L$ was arbitrary, $m_0\in J^{(r)}_{n,\ell}$. Thus $\alpha(J^{(r)}_{n,\ell})\le \deg m_0=\binom nr-\tr_r(n,\ell-1)$.
Combining the two inequalities proves the formula for $\alpha(J^{(r)}_{n,\ell})$.

By \cref{lem:cover-translation,prop:collapse},
\begin{align*}
\binom nr-\ex(n,\mathcal{K}_{\ell}^{(r)})
=\alpha(C_{\mathcal{K}_{\ell}^{(r)},n}) 
=\alpha(J^{(r)}_{n,\ell}) 
=\binom nr-\tr_r(n,\ell-1).
\end{align*}
Canceling $\binom nr$ gives the asserted consequence.
\end{proof}

\section*{Acknowledgements}

H.L. was supported by the National Natural Science Foundation of China (12501487), the China Scholarship Council, and the Institute for Basic Science (IBS-R029-C4).
X.L. was supported by the Excellent Young Talents Program (Overseas) of the National Natural Science Foundation of China.

\section*{Declaration on the use of AI}

The authors used generative AI tools to assist in discussing proof strategies, checking proofs, and improving exposition. All mathematical arguments, results, and conclusions were reviewed and verified by the authors.

\bibliographystyle{abbrv}
\bibliography{references}

\end{document}